\documentclass[12pt,a4paper]{article}
\usepackage[latin9]{inputenc}
\usepackage[active]{srcltx}
\usepackage{bm}
\usepackage{amsmath}
\usepackage{amssymb}

\makeatletter

\newcommand{\lyxaddress}[1]{
	\par {\raggedright #1
	\vspace{1.4em}
	\noindent\par}
}

\@ifundefined{date}{}{\date{}}

\usepackage[active]{srcltx}

\usepackage{amsfonts}
\usepackage{amsthm}
\usepackage{mathrsfs}

\newtheorem{theorem}{Theorem}
\newtheorem{proposition}[theorem]{Proposition}

\newtheorem*{lemma*}{Lemma}

\theoremstyle{remark}
\newtheorem{remark}[theorem]{Remark}
\newtheorem*{remark*}{Remark}
\newtheorem*{remarks*}{Remarks}
\newtheorem*{example*}{Example}
\newtheorem*{question*}{QUESTION}
\newtheorem*{conjecture*}{CONJECTURE}

\theoremstyle{definition}

\newtheorem*{definition*}{Definition}

\newtheorem*{notation*}{Notation}

\newcommand{\diag}{\mathop\mathrm{diag}\nolimits}

\newcommand{\Hn}{\mathop\mathrm{Hn}\nolimits}
\newcommand{\Hl}{\mathop\mathrm{H\ell}\nolimits}

\makeatother

\begin{document}
\title{A formula for the local Heun solution}
\author{Pavel \v{S}\v{t}ov\'\i\v{c}ek}
\maketitle

\lyxaddress{\hskip2.6em\emph{ Department of Mathematics, Faculty of Nuclear
Science, Czech Technical}\\
\hskip2.9em \emph{University in Prague, Trojanova 13, 120~00 Praha,
Czech Republic}}
\begin{abstract}
\noindent The local Heun solution is the unique solution to Heun's
equation which is analytic in the unit disk centered at $0\in\mathbb{C}$
and taking the value $1$ at the center of the disk. In this paper,
as an application of the theory of orthogonal polynomials, we are
able to express the coefficients in the corresponding power series
as finite multiple sums. In addition, the obtained formula can be
used to derive an explicit estimate on the coefficients giving a hint
on their asymptotic behavior for large indices.
\end{abstract}
\noindent \begin{flushleft}
\par\end{flushleft}

\section{Introduction}

Heun's equation is an analogue of the Gauss hypergeometric equation.
Both equations are second-order Fuchsian differential equations. While
the latter equation has three regular singular points, the former
equation has four such points. This article deals with the local Heun
solution which is a special function solving Heun's equation. Regarding
the respective history, terminology and notation we refer to the monograph
 \cite{Ronveaux}. One can also consult \cite[Sec.~15.3]{Erdelyi_etal}
or \cite[Chp.~31]{DLMF}.

Despite the long history since the seminal paper by Heun \cite{Heun}
the study of Heun's equation is far of being complete and suggests
some profound challenges. One of them is a more explicit description
of solutions to Heun's equation. This is in sharp contrast to the
hypergeometric equation. The principal complication in the case of
Heun's equation is caused by the extra singularity.

From the general theory of Fuchsian differential equations we know
that the general solution to Heun's equation in a neighborhood of
a regular singularity can be expressed as a linear combination of
two Frobenius series, and the coefficients in those series are generated
by a three-term recurrence \cite{Ronveaux}. To the best of author's
knowledge, there is a very limited progress in the attempt to describe
the coefficients more explicitly. In this respect we are aware just
of the paper \cite{Choun}.

An alternative attempt consists in expanding a solution as a series
in terms of simpler special functions, namely the hypergeometric functions
\cite{Ronveaux}. This approach has been developed in a number of
papers. Here we just mention the comparatively recent contribution
of this kind \cite{MelikdzhanianIshkhanyan} and the references therein.

One can also use the fact that Heun's equation is closely related
to the theory of orthogonal polynomials. For instance, some families
of orthogonal polynomials satisfy Heun\textquoteright s equation \cite{MagnusNdayiragijeRonveaux}.
Even more importantly, as emphasized by Valent in \cite{Valent-WorldSciPubl},
generating functions for some classes of orthogonal polynomials satisfy
Heun\textquoteright s equation. As an example of this relationship
let us mention the paper \cite{GoginHirvensalo}. In the present paper,
we apply this approach and combine it with some earlier results on
orthogonal polynomials presented in \cite{PS-JAT,PS-JAT2} in order
to obtain an explicit description of the coefficients in the Frobenius
series.

This article is organized as follows. In Section 2, we recall some
basic facts concerning Heun's equation which are relevant in the remainder
of the paper. Section 3 is devoted to the Green function of a Jacobi
matrix. An auxiliary result is proven here that turned out to be of
importance in the derivation of the sought formula for the coefficients
in the local Heun solution. The formula itself is proven in Section
4. As an application, the formula is used to derive an explicit estimate
on the coefficients. In Section 5, a sub-family of the local Heun
solutions is considered. In this particular case, we prove another
formula to hold. This formula is similar in nature but not identical
to that presented in Section 4.

\section{Preliminaries}

The canonical form of Heun's equation is
\begin{equation}
\frac{\text{d}^{2}F(z)}{\text{d}z^{2}}+\left(\frac{\gamma}{z}+\frac{\delta}{z-1}+\frac{\epsilon}{z-a}\right)\!\frac{\text{d}F(z)}{\text{d}z}+\frac{\alpha\beta z-q}{z(z-1)(z-a)}\,F(z)=0.\label{eq:Heun-eq-R}
\end{equation}
It depends on seven, generally complex, parameters $\alpha$, $\beta$,
$\gamma$, $\delta$, $\epsilon$, $a$, $q$, but only six of them
are independent. The singularities are located at the points $0$,
$1$, $a$ and $\infty$, and all of them are regular. The characteristic
exponents are respectively
\[
(0,1-\gamma),\ (0,1-\delta),\ (0,1-\epsilon)\ \text{and}\ (\alpha,\beta).
\]
One usually assumes that $|a|>1$. The parameters are subject to the
condition
\[
\gamma+\delta+\epsilon=\alpha+\beta+1,
\]
in order to guarantee regularity of the singularity at $\infty$.
The parameter $q$ is called the \emph{accessory parameter}.

Assuming that
\[
\gamma\neq0,-1,-2,\ldots,
\]
there exists, up to a constant multiplier, a unique Frobenius solution
at $z=0$ with the characteristic exponent $0$. It has the form of
a power series
\begin{equation}
F(z)=\sum_{n=0}^{\infty}c_{n}z^{n},\ \,c_{0}\neq0.\label{eq:heun-loc-ser}
\end{equation}
Under the assumption that $|a|>1$ the power series converges in the
disk $D_{1}$,
\[
D_{1}:=\{z\in\mathbb{C};\,|z|<1\}.
\]
The coefficients $c_{n}$ are determined by the three-term recurrence
\begin{eqnarray}
 & -qc_{0}+a\gamma c_{1}\,=\,0,\label{eq:Heun-recurr-0}\\
\noalign{\smallskip} & A_{n}c_{n-1}-(B_{n}+q)c_{n}+C_{n}c_{n+1}\,=\,0 & \text{for}\ n\geq1,\label{eq:Heun-recurr-1}
\end{eqnarray}
where
\begin{eqnarray}
A_{n} & := & (n-1+\alpha)(n-1+\beta),\nonumber \\
B_{n} & := & n\big((n-1+\gamma)(1+a)+a\delta+\epsilon\big),\label{eq:AnBnCn-orig}\\
C_{n} & := & (n+1)(n+\gamma)a.\nonumber 
\end{eqnarray}
To simplify the notation one frequently puts by convention $c_{-1}:=0$.

To remove the ambiguity of an arbitrary multiplier the solution (\ref{eq:heun-loc-ser})
is normalized by requiring $c_{0}=1$, that is $F(0)=1$. In \cite{Ronveaux},
this solution is called \emph{the local Heun solution} and is denoted
\[
\Hl(a,q;\alpha,\beta,\gamma,\delta;z).
\]

Following \cite{Valent-WorldSciPubl} we use the substitution
\begin{equation}
a:=\frac{1}{k^{2}}\,,\ q:=-\frac{s}{k^{2}}\,,\label{eq:a-q-Valent}
\end{equation}
with $k$ being restricted to the range $0<k<1$. Equation (\ref{eq:Heun-eq-R})
then takes the form
\begin{equation}
\frac{\text{d}^{2}F(z)}{\text{d}z^{2}}+\left(\frac{\gamma}{z}-\frac{\delta}{1-z}-\frac{\epsilon k^{2}}{1-k^{2}z}\right)\!\frac{\text{d}F(z)}{\text{d}z}+\frac{s+\alpha\beta k^{2}z}{z\,(1-z)(1-k^{2}z)}\,F(z)=0.\label{eq:Heun-eq-V}
\end{equation}
In \cite{Valent-WorldSciPubl}, the local Heun solution is then denoted
as 
\[
\Hn(k^{2},s;\alpha,\beta,\gamma,\delta;z):=\Hl\!\bigg(\frac{1}{k^{2}},-\frac{s}{k^{2}};\alpha,\beta,\gamma,\delta;z\bigg).
\]

The following facts are well known, see~\cite{Valent-WorldSciPubl}.
The power series (\ref{eq:heun-loc-ser}) can be rewritten in the
the form
\begin{equation}
\Hn(k^{2},s;\alpha,\beta,\gamma,\delta;z)=\sum_{n=0}^{\infty}(-1)^{n}\sqrt{\frac{\lambda_{0}\lambda_{1}\dots\lambda_{n-1}}{\mu_{1}\mu_{2}\dots\mu_{n}}}\,P_{n}(s+k^{2}\alpha\beta)z^{n},\label{eq:gen_func_onp}
\end{equation}
where $\{P_{n};\,n\geq0\}$ is a polynomial sequence fulfilling the
three-term recurrence
\[
a_{n-1}P_{n-1}(x)+(b_{n}-x)P_{n}(x)+a_{n}P_{n+1}(x)=0,\quad n\geq0,
\]
with the initial condition $P_{-1}(x)=0$, $P_{0}(x)=1$, and
\begin{equation}
a_{n}:=\sqrt{\lambda_{n}\mu_{n+1}}\,,\ \,b_{n}:=\lambda_{n}+\mu_{n}+\gamma_{n},\label{eq:alpha-beta-llbd-mu-gamma}
\end{equation}
\begin{equation}
\lambda_{n}:=k^{2}(n+\alpha)(n+\beta),\ \,\mu_{n}:=n\,(n+\gamma-1),\ \,\gamma_{n}:=(1-k^{2})\delta n.\label{eq:lbd-mu-gamma}
\end{equation}
For $\alpha,\beta,\gamma>0$, $\delta\in\mathbb{R}$, and $0<k<1$,
the polynomials $P_{n}$ are known to be \emph{orthonormal} with respect
to a unique probability measure on the real line, that is $\{P_{n}\}$
is an orthonormal polynomial sequence with the corresponding Hamburger
moment problem being determinate.

Following \cite{Valent-WorldSciPubl}, let us make yet another substitution
\begin{equation}
s:=w-k^{2}\alpha\beta.\label{eq:s-Valent}
\end{equation}
Then equation (\ref{eq:Heun-eq-V}) becomes
\begin{equation}
\frac{\text{d}^{2}F(z)}{\text{d}z^{2}}+\left(\frac{\gamma}{z}-\frac{\delta}{1-z}-\frac{\epsilon k^{2}}{1-k^{2}z}\right)\!\frac{\text{d}F(z)}{\text{d}z}+\frac{w-k^{2}\alpha\beta\,(1-z)}{z(1-z)(1-k^{2}z)}\,F(z)=0.\label{eq:Heun-eq-G-1}
\end{equation}
Hereafter, we deal with Heun's equation exactly in this form and write
\[
F(z):=\Hl\!\Big(\,\frac{1}{k^{2}},-\frac{w-k^{2}\alpha\beta}{k^{2}};\alpha,\beta,\gamma,\delta;z\Big).
\]

\section{The Green function}

\subsection{Generalities}

Let us recall the concept of the Green function related to a Jacobi
matrix \cite{Geronimo} that is substantial for our purposes. By a
Jacobi matrix $J$, we shall understand a real symmetric tridiagonal
semi-infinite matrix whose subdiagonal contains nonzero entries only.
We can even assume, without loss of generality, that the entries on
the subdiagonal are all positive. The sequences of entries on the
subdiagonal and the diagonal will be denoted $\{a_{n};\,n\geq0\}$
and $\{b_{n};\,n\geq0\}$, respectively. Thus $J$ has the form 
\[
J=\begin{pmatrix}b_{0} & a_{0}\\
a_{0} & b_{1} & a_{1}\\
 & a_{1} & b_{2} & \ddots\\
 &  & \ddots & \ddots
\end{pmatrix}.
\]
We shall refer to the sequences $\{a_{n}\}$ and $\{b_{n}\}$ as the
\emph{Jacobi parameters}.

The Green function of $J$ is a strictly lower triangular matrix $G$
which is a right inverse of $J$, that is
\[
JG=I\ \,(I\ \text{standing for the unit matrix}).
\]
Here we make use of the fact that the multiplication of band semi-infinite
matrices is well defined (in the purely algebraic sense). Let $J'$
be the matrix obtained from $J$ by removing the first column, and
$G'$ be the matrix obtained from $G$ by removing the first row (which
is just the zero row vector). Then $J'$ and $G'$ are both lower
triangular semi-infinite matrices, all entries on the diagonal of
$J'$ are positive, and $J'G'=I$. It readily follows that $G'=(J')^{-1}$,
with the inverse being well defined and unique. Hence the Green function
is defined unambiguously.

Recall, too, that to every Jacobi matrix $J$ there corresponds a
unique orthonormal polynomial sequence $\{P_{n}(x);\,n\geq0\}$ normalized
by requiring that $P_{0}(x)=1$. The polynomial sequence is determined
by a three-term recurrence which can be written in a compact form
as
\[
J\bm{P}(x)=x\,\bm{P}(x),
\]
where $\bm{P}(x)$ is a column vector,
\[
\bm{P}(x):=\big(P_{0}(x),P_{1}(x),P_{2}(x),\dots\,\big)^{T}.
\]
The usefulness of the Green function consists in the identity \cite[Theorem 6]{PS-JAT}
\begin{equation}
\bm{P}(x)=(I-xG)^{-1}\bm{P}(0).\label{eq:Px-P0}
\end{equation}
Here the inverse $(I-xG)^{-1}$ is well defined by the geometric series.
Moreover, one can express the Green function in terms of the values
$P_{n}(0)$, $n\geq0$, namely
\begin{equation}
G_{m,n}=\bigg(\,\sum_{j=n}^{m-1}\frac{1}{a_{j}P_{j}(0)P_{j+1}(0)}\,\bigg)P_{m}(0)P_{n}(0)\ \,\text{for}\ m>n\label{eq:Green-Pn0}
\end{equation}
(and $G_{m,n}=0$ for $m\leq$n) \cite[Eqs. (27),(29)]{PS-JAT}.

This means that provided the values $\{P_{n}(0)\}$ are known one
can derive a formula for all members $P_{n}(x)$ of the orthonormal
polynomial sequence. This approach has been already heavily exploited
in the follow up papers \cite{PS-JAT2,PS-Swiderski}.

A particularly important example of Jacobi matrices for which the
values $\{P_{n}(0)\}$ are well known are those Jacobi matrices which
are related to generators of Birth-Death processes. In that case 
\begin{equation}
a_{n}=\sqrt{\lambda_{n}\mu_{n+1}}\,,\ b_{n}=\lambda_{n}+\mu_{n},\label{eq:alpha-beta-llbd-mu}
\end{equation}
where $\{\lambda_{n}\}$ and $\{\mu_{n}\}$ are sequences fulfilling
$\lambda_{n}>0$, $\mu_{n+1}>0$ for all $n\geq0$, and $\mu_{0}=0$.
Then
\begin{equation}
P_{n}(0)=(-1)^{n}\,\sqrt{\frac{\lambda_{0}\lambda_{1}\ldots\lambda_{n-1}}{\mu_{1}\mu_{2}\ldots\mu_{n}}}\,,\ n\geq0.\label{eq:Pn0-birth-death}
\end{equation}

This equation is not difficult to verify. We can refer, for instance,
to the classical paper \cite[\S II.1]{Karlin_McGregor} where the
formula is already contained in somewhat different notation. In fact,
in this reference the even more general case with $\mu_{0}>0$ is
treated. A detailed discussion of this more general case is also provided
in \cite[Lemma 3.1]{PS-Swiderski}.

\subsection{The particular case $\delta=0$\label{sec:Fz-dlt-0}}

Comparing (\ref{eq:alpha-beta-llbd-mu}) to (\ref{eq:alpha-beta-llbd-mu-gamma})
we see that formula (\ref{eq:Pn0-birth-death}) could be used if we
had\linebreak{}
$\gamma_{n}=0$. Therefore, taking into account (\ref{eq:lbd-mu-gamma}),
we temporarily restrict ourselves to a five-parameter subfamily of
the local Heun solutions by letting
\[
\delta=0.
\]
Then, in view of (\ref{eq:lbd-mu-gamma}), (\ref{eq:alpha-beta-llbd-mu})
and (\ref{eq:Pn0-birth-death}),
\begin{equation}
a_{n}=k\sqrt{(n+1)(n+\alpha)(n+\beta)(n+\gamma)}\,,\ b_{n}=k^{2}(n+\alpha)(n+\beta)+n\,(n+\gamma-1),\label{eq:alpha-beta-delta-0}
\end{equation}
and
\begin{equation}
P_{n}(0)=(-1)^{n}k^{n}\,\sqrt{\frac{(\alpha)_{n}(\beta)_{n}}{n!(\gamma)_{n}}}\,.\label{eq:Pn0-dlt-0}
\end{equation}
Here and everywhere in what follows, $(x)_{n}$ denotes the Pochhammer
symbol,
\[
(x)_{n}:=x\,(x+1)(x+2)\cdots(x+n-1)
\]
for a non-negative integer $n$.

Using (\ref{eq:Green-Pn0}) and (\ref{eq:Pn0-birth-death}) we obtain,
in a straightforward way, the following result.

\begin{proposition} The Green function for the Jacobi parameters
$a_{n}$, $b_{n}$ given in (\ref{eq:alpha-beta-delta-0}) reads
\begin{equation}
G_{m,n}=(-1)^{m+n+1}k^{m+n-2}\sqrt{\frac{(\alpha)_{m}(\alpha)_{n}(\beta)_{m}(\beta)_{n}}{m!n!(\gamma)_{m}(\gamma)_{n}}}\,\sum_{j=n}^{m-1}k^{-2j}\,\frac{j!(\gamma)_{j}}{(\alpha)_{j+1}(\beta)_{j+1}},\ \,m>n\geq0,\label{eq:G-dlt-0}
\end{equation}
and $G_{m,n}=0$ for $0\le m\le n$. \end{proposition}

These results can be applied in the theory of the Heun functions.
Combining (\ref{eq:gen_func_onp}), (\ref{eq:s-Valent}) and (\ref{eq:Pn0-birth-death})
we get an expression for the local Heun solution,
\[
F(z)=\Hl\!\bigg(\frac{1}{k^{2}},-\frac{w}{k^{2}}+\alpha\beta;\alpha,\beta,\gamma,0;z\bigg)=\sum_{n=0}^{\infty}c_{n}z^{n}\ \,\text{where}\ c_{n}=P_{n}(0)P_{n}(w),
\]
and, in regard of (\ref{eq:Px-P0}),
\begin{eqnarray*}
c_{n} & = & \sum_{m=0}^{n}P_{n}(0)\big(G^{m}\bm{P}(0)\big)_{\!n}w^{m}\\
 & = & \sum_{m=0}^{n}\,\sum_{0\leq\ell_{1}<\ell_{2}<\,\ldots\,<\ell_{m}<n}\,P_{n}(0)G_{n,\ell_{m}}G_{\ell_{m},\ell_{m-1}}\cdots G_{\ell_{2},\ell_{1}}P_{\ell_{1}}(0)w^{m}.
\end{eqnarray*}
Here the last sum should be understood so that the summation runs
over all integer indices $\ell_{1},\ldots,\ell_{m}$ satisfying the
indicated inequalities. Substituting from (\ref{eq:Pn0-dlt-0}) and
(\ref{eq:G-dlt-0}) we obtain the formula
\begin{equation}
c_{n}=\frac{(\alpha)_{n}(\beta)_{n}}{n!(\gamma)_{n}}\,\bigg(k^{2n}+\sum_{m=1}^{n}g_{m,n}w^{m}\bigg)\label{eq:cn-dlt-0}
\end{equation}
where
\begin{equation}
g_{m,n}=(-1)^{m}k^{2n-2m}\,\sum_{0\leq\ell_{1}\leq j_{1}<\ell_{2}\leq j_{2}<\,\ldots\,<\ell_{m}\leq j_{m}<n}\,\,\prod_{s=1}^{m}\frac{j_{s}!\,(\gamma)_{j_{s}}(\alpha)_{\ell_{s}}(\beta)_{\ell_{s}}}{\ell_{s}!\,(\gamma)_{\ell_{s}}(\alpha)_{j_{s}+1}(\beta)_{j_{s}+1}}\,k^{-2(j_{s}-\ell_{s})}.\label{eq:gmn-dlt-0}
\end{equation}

\subsection{A diagonal perturbation of a Jacobi matrix}

Naturally, we wish to extend the formula (\ref{eq:cn-dlt-0}), (\ref{eq:gmn-dlt-0})
to the case of a general parameter $\delta$. Let $\check{J}$ refer
to the Jacobi matrix for a general $\delta$ and put $J:=\check{J}\big|_{\delta=0}$.
Then from the comparison of (\ref{eq:alpha-beta-llbd-mu-gamma}) to
(\ref{eq:alpha-beta-llbd-mu}) it is seen that $\check{J}-J$ is a
diagonal matrix,
\begin{equation}
D:=\check{J}-J=(1-k^{2})\delta\diag(0,1,2,\ldots).\label{eq:Jcheck-J}
\end{equation}
Fortunately, in such a case the respective Green functions are related
by a comparatively simple algebraic equation. We have the following
general result.

\begin{theorem} Let $J$ be a Jacobi matrix and $D$ be a real diagonal
semi-infinite matrix. Put
\[
\check{J}:=J+D.
\]
Let $G$ and $\check{G}$ be the Green functions for $J$ and $\check{J}$,
respectively. Denote by $\bm{P}(x)$ and $\check{\bm{P}}(x)$ the
orthonormal polynomial sequences associated with $J$ and $\check{J}$,
respectively, both being arranged into column vectors. Then
\begin{equation}
\check{G}=(I+GD)^{-1}G\label{eq:Gcheck}
\end{equation}
and
\begin{equation}
\check{\bm{P}}(x)=\big(I-G(xI-D)\big)^{-1}\bm{P}(0).\label{eq:Pcheck-x}
\end{equation}
In particular, 
\begin{equation}
\check{\bm{P}}(0)=(I+GD)^{-1}\bm{P}(0).\label{eq:Pcheck-0}
\end{equation}
\end{theorem}

\begin{remark*} Note that $(I+GD)^{-1}$ is well defined by the geometric
series owing to the fact that $G$ is strictly lower triangular and
$D$ is diagonal. For any couple of indices $m,n\geq0$, the series
\[
\big((I+GD)^{-1}G\big)_{m,n}=\sum_{k=0}^{\infty}(-1)^{k}\big((GD)^{k}G\big)_{m,n}
\]
has only finitely many nonzero terms. A similar remark applies in
other cases, too, in particular it is applicable to the RHS of (\ref{eq:Pcheck-x}).
\end{remark*}

\begin{proof} Firstly, we show (\ref{eq:Gcheck}). To this end, notice
that the RHS of (\ref{eq:Gcheck}) represents a strictly lower triangular
matrix. Now it suffices to verify that this matrix solves the equation
$\check{J}\check{G}=I$. Using that $JG=I$ and
\[
(I+GD)^{-1}G=G(I+DG)^{-1}
\]
we have
\[
\check{J}(I+GD)^{-1}G=(J+D)G(I+DG)^{-1}=(I+DG)(I+DG)^{-1}=I.
\]

Next we show (\ref{eq:Pcheck-0}). Notice that the RHS of (\ref{eq:Pcheck-0})
represents a column vector whose first entry equals $1$. It suffices
to verify that this column vector solves the equation $\check{J}\check{\bm{P}}(0)=0$.
We have
\[
J+D=J(I+GD)
\]
and consequently,
\[
(J+D)(I+GD)^{-1}=J.
\]
Whence
\[
\check{J}(I+GD)^{-1}\bm{P}(0)=(J+D)(I+GD)^{-1}\bm{P}(0)=J\bm{P}(0)=0.
\]

Finally, we show (\ref{eq:Pcheck-x}). We have
\begin{eqnarray*}
\check{\bm{P}}(x) & = & (I-x\check{G})^{-1}\check{\bm{P}}(0)\\
 & = & \big(I-x(I+GD)^{-1}G\big)^{-1}(I+GD)^{-1}\bm{P}(0)\\
 & = & \Big((I+GD)\big(I-x(I+GD)^{-1}G\big)\Big)^{\!-1}\bm{P}(0)\\
 & = & \big(I-G(xI-D)\big)^{-1}\bm{P}(0).
\end{eqnarray*}
This concludes the proof. \end{proof}

\section{The local Heun solution}

Next we focus on the local Heun solution with a general parameter
$\delta$. Again, let $\check{G}$ refer to the Green function for
the general case and $G$ refer to the Green function for the particular
case with $\delta=0$. Equation (\ref{eq:Gcheck}) in combination
with (\ref{eq:G-dlt-0}) and (\ref{eq:Jcheck-J}) can be used to get
a formula for $\check{G}$, and with such a formula at hand we can
mimic the procedure from Subsection \ref{sec:Fz-dlt-0} to get a formula
for the local Heun solution. However, here we avoid providing details
of such a derivation which could be somewhat tedious. Instead, in
Theorem \ref{thm:fin-formula}, we present the final formula for the
coefficients $c_{n}$ right away, and then we show that the power
series $F(z)=\sum_{n=0}^{\infty}c_{n}z^{n}$ actually defines the
sought local Heun solution.

Note that for a general parameter $\delta$ we have, see (\ref{eq:AnBnCn-orig}),
(\ref{eq:a-q-Valent}),
\begin{eqnarray}
A_{n} & = & (n-1+\alpha)(n-1+\beta),\nonumber \\
B_{n} & = & n\bigg(\!(n-1+\gamma)\Big(1+\frac{1}{k^{2}}\Big)+\frac{\delta}{k^{2}}+\alpha+\beta-\gamma-\delta+1\bigg)\!,\label{eq:AnBnCn-gen}\\
C_{n} & = & \frac{(n+1)(n+\gamma)}{k^{2}}\nonumber 
\end{eqnarray}
and, referring also to (\ref{eq:s-Valent}),
\begin{equation}
q=-\frac{w}{k^{2}}+\alpha\beta.\label{eq:q-gen}
\end{equation}

\begin{theorem}\label{thm:fin-formula} Suppose $\alpha,\beta,\gamma,\delta,w\in\mathbb{C}$,
$\gamma\notin\{0,-1,-2,\ldots\}$, and $k\in(0,1)$. Let us define,
for $m,n\in\mathbb{Z}$, $0\leq m\leq n$,
\[
f_{m,n}:=k^{2n}\ \ \text{if}\ m=0,
\]
and
\begin{eqnarray}
 &  & \hskip-1.5emf_{m,n}:=(-1)^{m}k^{2n-2m}\nonumber \\
 &  & \times\,\sum_{0\leq\ell_{1}\leq j_{1}<\ell_{2}\leq j_{2}<\,\ldots\,<\ell_{m}\leq j_{m}<n}\,\,\prod_{s=1}^{m}\frac{j_{s}!\,(\gamma)_{j_{s}}(\alpha)_{\ell_{s}}(\beta)_{\ell_{s}}}{\ell_{s}!\,(\gamma)_{\ell_{s}}(\alpha)_{j_{s}+1}(\beta)_{j_{s}+1}}\,k^{-2(j_{s}-\ell_{s})}\big(w-(1-k^{2})\delta\ell_{s}\big)\nonumber \\
\noalign{\smallskip} &  & \hskip28.5em\text{if}\ m\geq1,\label{eq:gmn-gen}
\end{eqnarray}
and put
\begin{equation}
c_{n}:=\frac{(\alpha)_{n}(\beta)_{n}}{n!(\gamma)_{n}}\,\sum_{m=0}^{n}f_{m,n}.\label{eq:cn-gen}
\end{equation}
Then the function
\begin{equation}
F(z):=\sum_{n=0}^{\infty}\,c_{n}z^{n}\ \,\text{for}\ z\in D_{1}\label{eq:F}
\end{equation}
is the local Heun solution, that is
\[
F(z)=\Hl\!\Big(\,\frac{1}{k^{2}},-\frac{w}{k^{2}}+\alpha\beta;\alpha,\beta,\gamma,\delta;z\Big).
\]
\end{theorem}

\begin{remark*} Again, the sum in (\ref{eq:gmn-gen}) should be understood
so that the summation runs over all integer indices $j_{1},\ldots,j_{m}$,
$\ell_{1},\ldots,\ell_{m}$ satisfying the indicated inequalities.
\end{remark*}

\begin{proof} It is convenient to extend the definition of the symbol
$f_{m,n}$. We put
\begin{equation}
f_{m,n}:=0\ \ \text{if}\ m=-1\ \text{or}\ m>n.\label{eq:gmn-conv}
\end{equation}
To show that $F(z)$, as defined in (\ref{eq:F}), is actually the
local Heun solution it suffices to verify that the sequence $\{c_{n}\}$
in (\ref{eq:cn-gen}) satisfies the recurrence relations (\ref{eq:Heun-recurr-0})
and (\ref{eq:Heun-recurr-1}), with $A_{n}$, $B_{n}$ and $C_{n}$
being given in (\ref{eq:AnBnCn-gen}). The normalization condition
$c_{0}=1$ is clearly fulfilled.

For $n=0$ we have
\[
c_{0}=f_{0,0}=1,\ c_{1}=\frac{\alpha\beta}{\gamma}\,(f_{0,1}+f_{1,1})=\frac{\alpha\beta k^{2}-w}{\gamma}
\]
and thus $(\ref{eq:Heun-recurr-0})$, with $q=(-w+k^{2}\alpha\beta)/k^{2}$
and $a=1/k^{2}$, holds. Further we assume that $n\geq1$.

By inspection of equation (\ref{eq:gmn-gen}) we observe that, for
$1\leq m\leq n$, the multiple sum in the expression for $f_{m,n}/k^{2}$
contains all terms from the multiple sum in the expression for $f_{m,n-1}$
plus the term with $j_{m}=n-1$. Hence
\begin{eqnarray*}
\frac{f_{m,n}}{k^{2}}-f_{m,n-1} & = & (-1)^{m}k^{2n-2m-2}\,\frac{(n-1)!(\gamma)_{n-1}}{(\alpha)_{n}(\beta)_{n}}\\
 &  & \hskip-3em\times\,\sum_{0\leq\ell_{1}\leq j_{1}<\,\ldots\,<\ell_{m-1}\leq j_{m-1}<\ell_{m}\leq n-1}\,\,\frac{(\alpha)_{\ell_{m}}(\beta)_{\ell_{m}}}{\ell_{m}!(\gamma)_{\ell_{m}}}\,k^{-2(n-1-\ell_{m})}\big(w-(1-k^{2})\delta\ell_{m}\big)\\
\noalign{\smallskip} &  & \qquad\qquad\ \times\,\prod_{s=1}^{m-1}\frac{j_{s}!\,(\gamma)_{j_{s}}(\beta)_{\ell_{s}}(\gamma)_{\ell_{s}}}{\ell_{s}!\,(\gamma)_{\ell_{s}}(\alpha)_{j_{s}+1}(\beta)_{j_{s}+1}}\,k^{-2(j_{s}-\ell_{s})}\big(w-(1-k^{2})\delta\ell_{s}\big).
\end{eqnarray*}
 Writing $\ell$ instead of $\ell_{m}$ and using once more the defining
relation (\ref{eq:gmn-gen}) we obtain
\[
\frac{f_{m,n}}{k^{2}}-f_{m,n-1}=-\frac{(n-1)!(\gamma)_{n-1}}{k^{2}(\alpha)_{n}(\beta)_{n}}\,\text{\ensuremath{\sum_{\ell=m-1}^{n-1}}}\frac{(\alpha)_{\ell}(\beta)_{\ell}}{\ell!(\gamma)_{\ell}}\big(w-(1-k^{2})\delta\ell\big)f_{m-1.\ell}.
\]
In regard of the convention (\ref{eq:gmn-conv}), this identity also
holds for $m=0$ (and $n\geq1$). From here we deduce that
\[
-\bigg(\frac{f_{m,n}}{k^{2}}-f_{m,n-1}\bigg)+\frac{(n+\alpha)(n+\beta)}{n\,(n+\gamma-1)}\bigg(\frac{f_{m,n+1}}{k^{2}}-f_{m,n}\bigg)=-\frac{w-(1-k^{2})\delta n}{k^{2}n\,(\gamma+n-1)}\,f_{m-1,n}.
\]
Let us multiply this equation by the factor $(\alpha)_{n}(\beta)_{n}/\big((n-1)!(\gamma)_{n-1}\big)$.
Then, in view of (\ref{eq:cn-gen}), the summation in $m$ from $0$
to $n+1$ gives (recall (\ref{eq:gmn-conv}))
\begin{eqnarray*}
 &  & -\,\frac{n\,(n+\gamma-1)}{k^{2}}\,c_{n}+(n+\alpha-1)(n+\beta-1)c_{n-1}\\
 &  & +\,\frac{(n+1)(n+\gamma)}{k^{2}}\,c_{n+1}-(n+\alpha)(n+\beta)c_{n}\\
 &  & +\,\frac{w-(1-k^{2})\delta n}{k^{2}}\,c_{n}\,=\,0.
\end{eqnarray*}
Taking into account (\ref{eq:AnBnCn-gen}), this equation in fact
coincides with (\ref{eq:Heun-recurr-1}). \end{proof}

From the general theory it follows that the series in (\ref{eq:F})
converges in the unit disk $D_{1}$. This can be also demonstrated
by an explicit estimate on the coefficients $c_{n}$, as defined in
(\ref{eq:cn-gen}). From the estimate one can deduce some information
about the asymptotic behavior of $c_{n}$ for large $n$, more precisely,
it is seen that $c_{n}$ grows at most polynomially in $n$ as $n\to\infty$.

\begin{proposition}\label{thm:estimate} The coefficients $c_{n}$,
as defined in (\ref{eq:cn-gen}), fulfill
\begin{equation}
\forall n\geq0,\ \,|c_{n}|\leq\frac{(d)_{n}^{\,2}}{n!|(\gamma)_{n}|}\,\exp\!\bigg(\frac{t}{1-k^{2}}\,\sum_{j=1}^{n}\frac{1}{j}\bigg)\!,\label{eq:estimate}
\end{equation}
where
\begin{eqnarray}
d & := & \max\{1,|\alpha|,|\beta|,|\gamma|\},\label{eq:d}\\
t & := & |w|+(1-k^{2})|\delta|.\label{eq:t}
\end{eqnarray}
\end{proposition}

\begin{proof} Suppose that $m\geq1$. For $0\leq\ell\leq j$ and
$x\in\mathbb{C}$, $(x)_{j}/(x)_{\ell}$ is a polynomial in $x$ with
nonnegative coefficients. Therefore, if $|x|\leq d$ then
\[
\bigg|\frac{(x)_{j}}{(x)_{\ell}}\bigg|\leq\frac{(|x|)_{j}}{(|x|)_{\ell}}\leq\frac{(d)_{j}}{(d)_{\ell}}\,.
\]
Note also that $j!=(1)_{j}$. Furthermore, if $0\leq\ell_{1}\leq j_{1}<\ell_{2}\leq j_{2}<\,\ldots\,<\ell_{m}\leq j_{m}<n$
then
\begin{equation}
(x)_{n}\,\prod_{s=1}^{m}\frac{(x)_{\ell_{s}}}{(x)_{j_{s}+1}}=\frac{x(x+1)(x+2)\,\cdots\,(x+n-1)}{(x+\ell_{1})\cdots(x+j_{1})(x+\ell_{2})\cdots(x+j_{2})\,\cdots\,(x+\ell_{m})\cdots(x+j_{m})}\label{eq:pole-b-g}
\end{equation}
is also a polynomial in $x$ with nonnegative coefficients. Therefore,
for $|x|\leq d$,
\[
\bigg|(x)_{n}\,\prod_{s=1}^{m}\frac{(x)_{\ell_{s}}}{(x)_{j_{s}+1}}\bigg|\leq(|x|)_{n}\,\prod_{s=1}^{m}\frac{(|x|)_{\ell_{s}}}{(|x|)_{j_{s}+1}}\leq(d)_{n}\,\prod_{s=1}^{m}\frac{(d)_{\ell_{s}}}{(d)_{j_{s}+1}}\,.
\]
Thus we find that
\[
\bigg|(\alpha)_{n}(\beta)_{n}\,\prod_{s=1}^{m}\frac{j_{s}!\,(\gamma)_{j_{s}}(\alpha)_{\ell_{s}}(\beta)_{\ell_{s}}}{\ell_{s}!\,(\gamma)_{\ell_{s}}(\alpha)_{j_{s}+1}(\beta)_{j_{s}+1}}\bigg|\leq(d)_{n}^{\,2}\,\prod_{s=1}^{m}\frac{1}{(j_{s}+d)^{2}}\leq(d)_{n}^{\,2}\,\prod_{s=1}^{m}\frac{1}{(j_{s}+1)^{2}}\,.
\]

In view of (\ref{eq:t}) we have, for $0\leq\ell\leq j$, 
\[
|w-(1-k^{2})\delta\ell|\leq t\,(\ell+1)\leq t\,(j+1).
\]

Using these estimates we get, for $m\geq1$,
\begin{eqnarray*}
|(\alpha)_{n}(\beta)_{n}f_{m,n}| & \leq & (d)_{n}^{\,2}\,k^{2n-2m}\,\sum_{0\leq j_{1}<j_{2}<\,\ldots\,j_{m}<n}\,\,\frac{k^{-2j_{1}-2j_{2}-\ldots-2j_{m}}}{(j_{1}+1)(j_{2}+1)\cdots(j_{m}+1)}\,t^{m}\\
\noalign{\smallskip} &  & \times\,\sum_{\ell_{1}=0}^{j_{1}}\,\sum_{\ell_{2}=j_{1}+1}^{j_{2}}\cdots\,\sum_{\ell_{m}=j_{m-1}+1}^{j_{m}}\,k^{2\ell_{1}+2\ell_{2}+\ldots+2\ell_{m}}\\
\noalign{\smallskip} & \leq & \frac{(d)_{n}^{\,2}t^{m}}{(1-k^{2})^{m}}\,\sum_{0\leq j_{1}<j_{2}<\,\ldots\,j_{m}<n}\,\,\frac{k^{2n-2j_{m}-2}}{(j_{1}+1)(j_{2}+1)\cdots(j_{m}+1)}\\
\noalign{\smallskip} & \leq & \frac{(d)_{n}^{\,2}\,t^{m}}{(1-k^{2})^{m}}\,\sum_{1\leq j_{1}<j_{2}<\,\ldots\,j_{m}\leq n}\,\,\frac{1}{j_{1}j_{2}\cdots j_{m}}\\
\noalign{\smallskip} & \leq & \frac{(d)_{n}^{\,2}\,t^{m}}{(1-k^{2})^{m}m!}\,\bigg(\sum_{j=1}^{n}\frac{1}{j}\bigg)^{\!m}.
\end{eqnarray*}
For $m=0$ we have
\[
|(\alpha)_{n}(\beta)_{n}f_{0,n}|\leq(d)_{n}^{\,2}.
\]
Inequality (\ref{eq:estimate}) follows. \end{proof}

\section{The five-parameter subfamily for $\delta=\beta+1$}

Finally we wish to point out that for the five-parameter subfamily
of the local Heun solutions obtained by letting
\[
\delta=\beta+1
\]
it is possible to derive another formula, although of similar nature
as that presented in Theorem \ref{thm:fin-formula}. In (\ref{eq:Heun-eq-V}),
let us now make the substitution
\[
s:=w-\beta\gamma.
\]
Then the equation becomes
\begin{equation}
\frac{\text{d}^{2}F(z)}{\text{d}z^{2}}+\left(\frac{\gamma}{z}-\frac{\delta}{1-z}-\frac{\epsilon k^{2}}{1-k^{2}z}\right)\!\frac{\text{d}F(z)}{\text{d}z}+\frac{w-\beta\gamma+\alpha\beta k^{2}z}{z(1-z)(1-k^{2}z)}\,F(z)=0.\label{eq:Heun-eq-G}
\end{equation}
In this section, we focus on Heun's equation written in this form.
The local Heun solution will be redenoted as $\tilde{F}(z)$.

\begin{theorem} Suppose ~$\alpha,\beta,w\in\mathbb{C}$, ~$\gamma\in\mathbb{C\setminus}\{0,-1,-2,\ldots\}$,
~$k\in(0,1)$, and let\linebreak{}
$\delta:=\beta+1$. Further let us define, for $m,n\in\mathbb{Z}$,
$0\leq m\leq n$,
\[
\tilde{f}_{m,n}=1\ \ \text{if}\ m=0,
\]
and
\begin{equation}
\tilde{f}_{m,n}:=(-1)^{m}\ \sum_{0\leq\ell_{1}\leq j_{1}<\ell_{2}\leq j_{2}<\,\ldots\,<\ell_{m}\leq j_{m}<n}\,\,\prod_{s=1}^{m}\frac{j_{s}!\,(\alpha)_{j_{s}}(\beta)_{\ell_{s}}(\gamma)_{\ell_{s}}}{\ell_{s}!\,(\alpha)_{\ell_{s}}(\beta)_{j_{s}+1}(\gamma)_{j_{s}+1}}\,k^{2(j_{s}-\ell_{s})}\ \ \text{if}\ m\geq1.\label{eq:gmn}
\end{equation}
Then the function
\begin{equation}
\tilde{F}(z):=\sum_{m=0}^{n}\frac{(\beta)_{n}}{n!}\,\tilde{f}_{m,n}w^{m}z^{n}\ \,\text{for}\ z\in D_{1}\label{eq:G}
\end{equation}
is the local Heun function for the respective parameters, that is
\[
\tilde{F}(z)=\Hl\!\Big(\,\frac{1}{k^{2}},-\frac{w-\beta\gamma}{k^{2}};\alpha,\beta,\gamma,\beta+1;z\Big).
\]
\end{theorem}

\begin{remark} 1) The RHS of (\ref{eq:gmn}) has poles in the variable
$\beta$ for non-positive integer values of this parameter. But these
singularities are compensated by the factor $(\beta)_{n}$ on the
RHS of (\ref{eq:G}), that is the expression $(\beta)_{n}\tilde{f}_{m,n}$
is free of any singularity in $\beta$, see equation (\ref{eq:pole-b-g})
below.

2) Note that
\[
\tilde{F}(z)\big|_{w=0}=(1-z)^{-\beta}.
\]
\end{remark}

\begin{proof} We claim that $\tilde{F}(z)$ is actually the local
Heun solution. This means that the sequence $\{c_{n};\,n\geq0\}$,
\begin{equation}
c_{n}:=\frac{(\beta)_{n}}{n!}\,\sum_{m=0}^{n}\tilde{f}_{m,n}w^{m},\label{eq:cn}
\end{equation}
satisfies the recurrence relations (\ref{eq:Heun-recurr-0}), (\ref{eq:Heun-recurr-1}),
with
\begin{eqnarray}
A_{n} & = & (n-1+\alpha)(n-1+\beta),\nonumber \\
B_{n} & = & n\bigg(\!(n-1+\gamma)\Big(1+\frac{1}{k^{2}}\Big)+\frac{\beta+1}{k^{2}}+\alpha-\gamma\bigg)\!,\label{eq:AnBnCn}\\
C_{n} & = & \frac{(n+1)(n+\gamma)}{k^{2}}.\nonumber 
\end{eqnarray}

Equation (\ref{eq:Heun-recurr-0}) amounts to $\gamma\beta\tilde{f}_{1,1}=-1$
which is true. Further we assume that $n\geq1$.

It is convenient to extend the definition of the symbol $\tilde{f}_{m,n}$.
We put
\[
\tilde{f}_{m,n}=0\ \ \text{if}\ m=-1\ \text{or}\ m>n.
\]
Substituting (\ref{eq:cn}) for $c_{n}$ in (\ref{eq:Heun-recurr-1})
and equating the coefficients at the same powers of $w$ gives
\begin{eqnarray*}
 &  & A_{n}\tilde{f}_{m,n-1}-\Big(B_{n}+\frac{\beta\gamma}{k^{2}}\,\Big)\frac{\beta+n-1}{n}\,\tilde{f}_{m,n}+\frac{\beta+n-1}{k^{2}n}\,\tilde{f}_{m-1,n}\\
\noalign{\smallskip} &  & +\,C_{n}\,\frac{(\beta+n-1)(\beta+n)}{n(n+1)}\,\tilde{f}_{m,n+1}=0,\ \ m\geq0.
\end{eqnarray*}
In view of (\ref{eq:AnBnCn}), these equations can be rewritten as
\begin{equation}
-n(n-1+\alpha)(\tilde{f}_{m,n}-\tilde{f}_{m,n-1})+\frac{1}{k^{2}}\,\tilde{f}_{m-1,n}+\frac{(n+\beta)(n+\gamma)}{k^{2}}\,(\tilde{f}_{m,n+1}-\tilde{f}_{m,n})=0,\ \ m\geq0.\label{eq:aux-g-diff}
\end{equation}
Clearly, (\ref{eq:aux-g-diff}) holds for $m=0$.

In regard of (\ref{eq:gmn}), for $m\geq1$ (and still $n\geq1$)
we have
\begin{eqnarray*}
\tilde{f}_{m,n}-\tilde{f}_{m,n-1} & = & (-1)^{m}\,\frac{(n-1)!(\alpha)_{n-1}}{(\beta)_{n}(\gamma)_{n}}\\
 &  & \times\,\sum_{0\leq\ell_{1}\leq j_{1}<\,\ldots\,<\ell_{m-1}\leq j_{m-1}<\ell_{m}\leq n-1}\,\,\frac{(\beta)_{\ell_{m}}(\gamma)_{\ell_{m}}}{\ell_{m}!(\alpha)_{\ell_{m}}}\,k^{2(n-1-\ell_{m})}\\
\noalign{\smallskip} &  & \qquad\qquad\qquad\quad\times\,\prod_{s=1}^{m-1}\frac{j_{s}!\,(\alpha)_{j_{s}}(\beta)_{\ell_{s}}(\gamma)_{\ell_{s}}}{\ell_{s}!\,(\alpha)_{\ell_{s}}(\beta)_{j_{s}+1}(\gamma)_{j_{s}+1}}\,k^{2(j_{s}-\ell_{s})}\,.
\end{eqnarray*}
From here we infer that
\[
\tilde{f}_{m,n}-\tilde{f}_{m,n-1}=-\,\frac{(n-1)!(\alpha)_{n-1}}{(\beta)_{n}(\gamma)_{n}}\,\sum_{\ell=m-1}^{n-1}\,\frac{(\beta)_{\ell}(\gamma)_{\ell}}{\ell!(\alpha)_{\ell}}\,k^{2(n-1-\ell)}\,\tilde{f}_{m-1,\ell}.
\]
Finally we find that
\begin{eqnarray*}
 &  & -n(n-1+\alpha)(\tilde{f}_{m,n}-\tilde{f}_{m,n-1})+\frac{(n+\beta)(n+\gamma)}{k^{2}}\,(\tilde{f}_{m,n+1}-\tilde{f}_{m,n})\\
 &  & =\,\frac{n!(\alpha)_{n}}{(\beta)_{n}(\gamma)_{n}}\,\sum_{\ell=m-1}^{n-1}\,\frac{(\beta)_{\ell}(\gamma)_{\ell}}{\ell!(\alpha)_{\ell}}\,k^{2(n-1-\ell)}\,\tilde{f}_{m-1,\ell}\\
 &  & \quad-\,\frac{n!(\alpha)_{n}}{(\beta)_{n}(\gamma)_{n}}\,\sum_{\ell=m-1}^{n}\,\frac{(\beta)_{\ell}(\gamma)_{\ell}}{\ell!(\alpha)_{\ell}}\,k^{2(n-1-\ell)}\,\tilde{f}_{m-1,\ell}\\
 &  & =\,-\frac{1}{k^{2}}\,\tilde{f}_{m-1,n}.
\end{eqnarray*}
Hence (\ref{eq:aux-g-diff}) holds for $m\geq1$, too. This concludes
the proof. \end{proof}

In this particular case, too, one can derive an estimate on the coefficients
$c_{n}$, very similarly as it has been done in Proposition \ref{thm:estimate}.
But the derivation is rather routine and brings in no principally
new ideas. Therefore, we omit the details and just state the result.

\begin{proposition} The coefficients $c_{n}$, as defined in (\ref{eq:cn}),
fulfill
\[
\forall n\geq0,\ \,|c_{n}|\leq\exp\!\bigg(\frac{\pi^{2}\,|w|}{6\,(1-k^{2})}\bigg)\frac{(d)_{n}^{\,2}}{n!\,|(\gamma)_{n}|},
\]
where $d$ is the same as in (\ref{eq:d}). \end{proposition}

\end{document}